\documentclass[10pt]{article}
\usepackage{graphicx} 
\usepackage{amssymb}
\usepackage{amsmath}
\usepackage{bbold}
\usepackage{geometry}
\usepackage{algorithm}
\usepackage{algpseudocode}
\usepackage{pgfplots}
\pgfplotsset{compat=1.8}
\usepgfplotslibrary{statistics}

\usepackage[colorlinks=true, allcolors=blue]{hyperref}  
\usepackage{cleveref} 
\usepackage{hyperref}

\usepackage{amsthm}
\usepackage{xcolor}
\usepackage{multirow}
\usepackage{authblk}

\let\oldtextit\textit 
\renewcommand\emph[1]{\oldtextit{\color{RoyalBlue}#1}}
\definecolor{RoyalBlue}{cmyk}{1, 0.50, 0, 0}
\makeatletter
\def\thanks#1{\protected@xdef\@thanks{\@thanks
        \protect\footnotetext{#1}}}
\makeatother
\theoremstyle{definition}

\newtheorem{theorem}{Theorem}

\newtheorem{lemma}[theorem]{Lemma}
\newtheorem{prop}[theorem]{Proposition}

\title{A priori bounds for certified Krawczyk homotopy tracking}
\author{Kisun Lee\thanks{
\hspace*{-1.8em}School of Mathematical and Statistical Science, Clemson University, 220 Parkway Drive, Clemson, SC 29634 
\\
\texttt{email:}\href{mailto:kisunl@clemson.edu}{kisunl@clemson.edu}, \texttt{ URL:}\href{https://klee669.github.io}{https://klee669.github.io}}}
\date{}

\begin{document}

\maketitle
\begin{abstract}
We establish the first complexity analysis for Krawczyk-based certified homotopy tracking. It consists of explicit a priori stepsize bounds ensuring the success of the Krawczyk test, and an iteration count bound proportional to the weighted length of the solution path. Our a priori bounds reduce the overhead of interval arithmetic, resulting in fewer iterations than previous methods. Experiments using a proof-of-concept implementation validate the results.
\end{abstract}

\section{Introduction}

\emph{Homotopy path tracking} is a method for solving systems of 
nonlinear equations. Given a target system $F:\mathbb{C}^n \to \mathbb{C}^n$ to solve, the method introduces a start system $G:\mathbb{C}^n \to \mathbb{C}^n$ with known solutions and defines a homotopy $F_t: \mathbb{C}^n \times [0,1] \to \mathbb{C}^n$ such that $F_0 = G$ and $F_1 = F$. If both $F=0$ and $G=0$ have finitely many nonsingular solutions, one can track a solution path $x(t)$ from $t=0$ to $t=1$ 
starting at a solution to $G=0$.

Tracking a solution path is carried out using a numerical 
\emph{predictor-corrector method}~\cite[Chapter 2.3]{sommese2005numerical}. Given a time $t_0 \in [0,1]$ and an approximation $x_0$ to a solution to $F_{t_0}=0$, the method computes the next time step $t_1 \in (t_0, 1]$ along with an approximation $x_1$ to a solution of $F_{t_1}=0$. Repeating this process yields a sequence $0 = t_0 < t_1 < \dots < t_k = 1$ 
with corresponding approximations $x_0, x_1, \dots, x_k$.

\emph{Certified homotopy path tracking} algorithms rigorously guarantee correctness by enclosing the path $x(t)$ within a compact region where $F_t = 0$ has a unique solution for each $t \in [0,1]$. This approach has been studied using Smale's alpha theory \cite[Chapter 8]{blum2012complexity} in various formulations \cite{beltran2012certified,beltran2013robust,hauenstein2014posteriori,hauenstein2016certified} and using interval arithmetic 
\cite{kearfott1994interval,van2011reliable,xu2018approach}. Recently, certified tracking using the interval arithmetic has significantly improved \cite{duff2024certified,guillemot2025certified,guillemot2024validated} and has extended to various applications \cite{brazelton2024monodromy,burr2025certified, duff2025numerically,rodriguez2017numerical}. These approaches compute a sequence of discrete time steps $0 = t_0 < t_1 < \cdots < t_k = 1$ 
along with interval boxes $I_i$ that contain the solution path $x(t)$ over each subinterval $[t_{i-1}, t_i]$.

This paper establishes the first complexity analysis for Krawczyk-based certified tracking. We achieve this through two key results. First, we derive explicit \textit{a priori} stepsize bounds that guarantee the success of the Krawczyk test (\Cref{thm:stepsize_bound}). Unlike previous methods that repeatedly evaluate the Krawczyk operator to determine safe stepsizes, our bound depends only on local geometry that can be computed once per step. Second, we provide the total iteration bound of the Krawczyk homotopy in terms of the weighted length of the solution path (\Cref{thm:complexity}).

The significance of this analysis extends to other certified algorithms in algebraic geometry. For instance, certified approximation of algebraic varieties can be achieved either through path tracking-based methods \cite{burr2025certified,martin2013certified} or subdivision-based approaches \cite{burr2020complexity,burr2025certified2,plantinga2007isotopic}. Our complexity analysis enables the analysis of path tracking-based methods, allowing direct comparison with subdivision-based paradigms. More broadly, our approach can be adapted to the complexity analysis of other algorithms employing Krawczyk certification. In addition, on the practical side, our \textit{a priori} stepsize bounds demonstrate a core principle: reducing the overhead of interval arithmetic improves tracking efficiency. Experiments confirm iteration reductions by eliminating repeated Krawczyk operator evaluation (see \Cref{sec:efficiency}).

We begin with introducing preliminaries in \Cref{sec:preliminaries}. In \Cref{sec:a_priori_predictor}, we derive \textit{a priori} stepsize bounds and present a modified tracking algorithm. \Cref{sec:complexity_analysis} establishes iteration count bounds based on these stepsizes. \Cref{sec:experiments} validates our results with experiments.

\section{Preliminaries}\label{sec:preliminaries}
Certified homotopy tracking relies on interval arithmetic to rigorously bound numerical errors, and the Krawczyk operator to certify solution existence and uniqueness. We review these concepts and their integration into the path tracking algorithm.

\subsection{Interval arithmetic}
Interval arithmetic extends standard arithmetic to intervals, yielding conservative enclosures of exact results. For an arithmetic operator $\odot$ and intervals $I, J$, we define 
\[I\odot J:=\{x\odot y\mid x\in I, y\in J\}.\] For instance, $[a,b] + [c,d] = [a+ c,b+ d]$. Explicit formulas for basic operations can be found in \cite{moore2009introduction}.

While intervals are naturally defined over $\mathbb{R}$, we extend interval arithmetic to $\mathbb{C}$ by treating real and imaginary parts independently. A complex interval $I$ takes the form $I = \Re(I) + i\Im(I)$ where $\Re(I), \Im(I) \subset \mathbb{R}$ are real intervals. For two complex intervals $I$ and $J$, we define
\begin{align*}
I+J &= (\Re(I)+\Re(J)) + i(\Im(I)+\Im(J))\\
I-J &= (\Re(I)-\Re(J)) + i(\Im(I)-\Im(J))\\
I\cdot J &= (\Re(I)\cdot\Re(I)-\Im(J)\cdot \Im(J)) + i(\Re(I)\cdot\Im(J)+\Im(I)\cdot\Re(J))\\
I/J &= \frac{\Re(I)\cdot\Re(J)+\Im(I)\cdot\Im(J)}{\Re(I)\cdot \Re(I)+\Im(J)\cdot \Im(J)} + i\frac{\Im(I)\cdot\Re(J)-\Re(I)\cdot\Im(J)}{\Re(I)\cdot \Re(I)+\Im(J)\cdot \Im(J)},\quad 0\not\in J.
\end{align*}
Throughout this paper, we work with complex intervals unless mentioned otherwise.

Let $I=(I_1,\dots, I_n)$ be an $n$-dimensional interval vector in $\mathbb{C}^n$. For a function $f:\mathbb{C}^n\rightarrow \mathbb{C}$, an \emph{interval enclosure} of $f$ over $I$ is an interval $\square f(I) \subset \mathbb{C}$ satisfying
\[\square f(I)\supset \{f(x)\mid x\in I\}.\]
That is, $\square f(I)$ is an interval that encloses the image of $f$ on $I$. Note that $\square f(I)$ is not unique, as it depends on how interval arithmetic is applied to evaluate $f$. 

For an interval vector $I = (I_1,\dots, I_n)$ in $\mathbb{C}^n$, we define
\[\|I\| = \max_{1\leq i\leq n}\max\{|\Re(I_i)|, |\Im(I_i)|\}\]
where $|J| = \max\limits_{a\in J}|a|$ denotes the magnitude of an interval $J$ over $\mathbb{R}$.
This is analogous to the $\infty$-norm when treating $\mathbb{C}^n$ as $\mathbb{R}^{2n}$.

\subsection{Krawczyk method}

The Krawczyk method combines interval arithmetic with Newton's method to certify the existence and uniqueness of solutions to square systems. Although the Krawczyk method is typically presented for real variables, we work in the complex setting as analyzed in \cite{burr2019effective}.

Let $F:\mathbb{C}^n\rightarrow\mathbb{C}^n$ be a square polynomial system and $x\in \mathbb{C}^n$ be a point. For $r > 0$ and an invertible matrix $Y \in \mathbb{C}^{n\times n}$, the \emph{Krawczyk operator} is
\[K(F,x,r,Y) := -YF(x) + (Id - Y\square JF(x+rB))rB\]
where $B = ([-1,1]+i[-1,1])^n$ denotes the unit box in $\mathbb{C}^n$.
We have the following theorem:

\begin{theorem}[see \cite{burr2025certified,guillemot2024validated,krawczyk1969newton}]\label{thm:Krawczyk}
If $K(F,x,r,Y) \subset r\rho B$ for some $r > 0$ and invertible matrix $Y \in \mathbb{C}^{n\times n}$, then there exists a unique solution $x^\star$ to $F=0$ in $x + rB$. Moreover, $\|x-x^\star\|\leq r\rho$ and the quasi-Newton iteration $x-YF(x)$ is $\rho$-Lipschitz continuous.
\end{theorem}
When this holds, we say that $x$ is a \emph{$\rho$-approximate solution} with \emph{certification radius} $r$. 

In practice, the matrix $Y$ is chosen to minimize $\|Id - Y\square JF(x+rB)\|$. The standard choice is $Y = JF(x)^{-1}$, which we adopt for our \textit{a priori} stepsize and complexity analysis in \Cref{sec:a_priori_predictor,sec:complexity_analysis}. Additionally, since exact evaluation of $F(x)$ is often infeasible, one typically replaces $F(x)$ with its interval enclosure $\square F(x)$.

To verify the containment $K(F,x,r,Y)\subset r\rho B$, one checks the inequality
\begin{equation}\label{eq:krawczyk_condition}
    \|K(F,x,r,Y)\| < r\rho.
\end{equation}
Since $\|K(F,x,r,Y)\|$ depends on how interval arithmetic is applied when evaluating $F$ and $JF$, different evaluation strategies may yield different results. Consequently, overly conservative enclosures may cause certification to fail even when $x$ is close to an actual solution.

\subsection{Krawczyk homotopy}

The Krawczyk method extends to certified homotopy tracking. The key idea is to apply the Krawczyk operator to a homotopy $F_t=0$ over an interval $T \subset [0,1]$, certifying the existence and uniqueness of a solution $x(t)$ for each $t \in T$. We now describe the algorithm based on \cite{duff2024certified,guillemot2024validated}.

Given a polynomial system $F$ and a $\tau$-approximate solution $x$ with $0<\tau<1$, the core task in the Krawczyk homotopy is to refine $x$ into a $\rho$-approximate solution for a desired threshold $0<\rho<\tau$. The refinement strategy repeats quasi-Newton steps $x \mapsto x - YF(x)$ and radius adjustments until the Krawczyk test $\|K(F,x,r,Y)\|<r\rho$ is satisfied. We present \Cref{algo:meta_refine} for this task.

\begin{algorithm}[ht]
	\caption{RefineSolution~\cite[Algorithm 2]{guillemot2024validated}}
 \label{algo:meta_refine}
\begin{algorithmic}[1]
\Require  
a polynomial system $F:\mathbb{C}^n\rightarrow \mathbb{C}^n$, a $\tau$-approximate solution $x$ with certification radius $r$ and matrix $Y \in \mathbb{C}^{n\times n}$, and target threshold $\rho\in (0,\tau)$.
\Ensure a $\rho$-approximate solution $\tilde{x}$ with certification radius $\tilde{r}$ and matrix $\tilde{Y}$.
\State{Initialize $\tilde{x} = x$, $\tilde{r} = r$, $\tilde{Y} = Y$}
\While{$\|K(F,\tilde{x},\tilde{r},\tilde{Y})\| \geq \tilde{r}\rho$} \Comment{Refinement loop}
\If{$\|Y F(\tilde{x})\| \leq \frac{1}{8}(1-\rho)\tau \tilde{r}$} 
\State{$\tilde{r} = \frac{1}{2}\tilde{r}$} \Comment{Shrink radius if the solution is accurate enough}
\Else
\State{$\tilde{x} = \tilde{x} - YF(\tilde{x})$} \Comment{Newton step}
\EndIf
\State{$\tilde{Y} = JF(\tilde{x})^{-1}$} \Comment{Update Jacobian}
\EndWhile
\While{$2\tilde{r} \leq 1$ and $\|K(F,\tilde{x},2\tilde{r},\tilde{Y})\| < 2\tilde{r}\rho$}\label{line:while_loop_in_Refine} \Comment{Expand radius}
\State{$\tilde{r} = 2\tilde{r}$}
\EndWhile
\State{\Return $\tilde{x}, \tilde{r}, \tilde{Y}$}
 \end{algorithmic}
 \end{algorithm}
The second loop maximizes the certification radius while preserving the Krawczyk condition. A larger radius reduces the number of tracking steps needed, as a small radius increases the risk of the solution path leaving the interval box. The condition $2\tilde{r} \leq 1$ in this loop
is not necessary for the algorithm's correctness, but ensures $r \leq 1$
for the complexity analysis in \Cref{sec:complexity_analysis}. 
In practice, the bound $r\leq 1$ is naturally satisfied in most tracking scenarios.

The Krawczyk homotopy operates in two phases: \textbf{refine tight with $\rho$, proceed with $\tau$}. At time $t$, we refine the solution to satisfy $\|K(F_t,x,r,Y)\| \leq r\rho$. We then proceed over the interval $T=[t,t+h]$ with $\|K(F_T,x,r,Y)\| \leq r\tau$ where $\tau>\rho$.
\Cref{algo:track} provides the algorithm for the Krawczyk homotopy.

\begin{algorithm}[ht]
	\caption{Track~\cite[Algorithm 3]{guillemot2024validated}}
 \label{algo:track}
\begin{algorithmic}[1]
\Require  
a homotopy $F_t:\mathbb{C}^n\times [0,1]\rightarrow \mathbb{C}^n$ with solution path $x(t)$ satisfying $F_t(x(t))=0$, 
a point $x$ approximating $x(0)$ with radius $r$ and matrix $Y$, 
and thresholds $0<\rho<\tau<1$.
\Ensure a $\rho$-approximate solution to $F_1=0$.
\State{Initialize $t=0, dt=1$.}
\While{$t<1$} 
\State{$x,r,Y=\text{RefineSolution}(F_t,x,r,Y,\rho)$} \label{line:refine}\Comment{Refine $x$ with $\rho$}
\State{$dt=2dt$}
\State{$T=[t,t+dt]$}
\While{$\|K(F_T,x,r,Y)\|\geq r\tau$}
\Comment{Proceed with $\tau$}
\State{$dt=\frac{1}{2}dt$} \State{$T=[t,t+dt]$} 
\EndWhile
\State{$t=t+dt$}
\EndWhile
\State{$x,r,Y=\text{RefineSolution}(F_1,x,r,Y,\rho)$} 
\State{\Return $x$}
 \end{algorithmic}
 \end{algorithm}

Note that \Cref{algo:track} employs an adaptive strategy; the Krawczyk operator is repeatedly computed to determine a stepsize $dt$ for the success of the Krawczyk method. Also, it uses a constant predictor, meaning that the approximate solution $x$ at time $t$ is used directly as the initial guess at $t+dt$. Although higher-order predictors can improve practical performance \cite{duff2024certified,guillemot2024validated}, we focus on the constant predictor case because it provides the essential complexity bounds and serves as a theoretical baseline for more sophisticated methods.

\section{\textit{A priori} stepsize bounds}\label{sec:a_priori_predictor}

In this section, we derive an \textit{a priori} bound on the stepsize for the Krawczyk homotopy. This bound determines how far we can proceed while guaranteeing the Krawczyk method succeeds. Our result is the Krawczyk homotopy analogue of the stepsize bounds for alpha theoretic certified tracking studied in \cite{beltran2013robust,hauenstein2016certified,shub2009complexity,shub1993complexity}.

\subsection{Affine linear parameter homotopies}
Throughout this paper, we work with affine linear parameter homotopies. It is a class of homotopies that arises naturally in various applications and whose explicit structure makes them amenable to complexity analysis.

Consider a homotopy $F_t(x) = F(x; p(t))$ where the parameter varies affinely $p(t) = tp_0 + (1-t)p_1$. We call such $F_t$ an \emph{affine linear} parameter homotopy. We can decompose
\[F_t(x) = F^{(1)}(x; p(t)) + F^{(2)}(x)\]
where $F^{(1)}(x; p(t))$ contains all monomials involving parameters and $F^{(2)}(x)$ collects parameter-free monomials. The affine structure gives us explicit control over how $F_t$ varies with $t$. Specifically, for any stepsize $dt$:
\begin{align*}
    F_{t+dt}(x) &= F^{(1)}(x; (p_1-p_0)(t+dt) + p_0) + F^{(2)}(x)\\
    &= F^{(1)}(x; (p_1-p_0)t + p_0) + F^{(1)}(x; (p_1-p_0)dt) + F^{(2)}(x)\\
    &= F_t(x) + F^{(1)}(x; (p_1-p_0)dt)\\
    &= F_t(x) + dt  F^{(1)}(x; p_1-p_0).
\end{align*}
This linearity extends to all derivatives, that is, $J^k F_{t+dt} = J^k F_t + dt J^k F^{(1)}(x; p_1-p_0)$ for all $k$. 
Since the parameter $p_1 - p_0$ is fixed throughout, we suppress it from the notation and simply write $F^{(1)}(x)$ for $F^{(1)}(x; p_1-p_0)$.

\subsection{Derivation}

We now derive explicit stepsize bounds for the Krawczyk homotopy. Consider an affine linear parameter homotopy $F_t=0$ and fix $t_0 \in [0,1]$. Let $x$ be a $\rho$-approximate solution to $F_{t_0}=0$ with certification radius $r$ obtained from \Cref{algo:meta_refine}. We aim to find the maximum stepsize $dt$ such that $x$ remains a $\tau$-approximate solution to $F_{[t_0,t_0+dt]}=0$, allowing the looser threshold $\tau > \rho$.

Our guiding principle is \textbf{refine tight with $\rho$, proceed with $\tau$}. To make this precise, we translate it into two Krawczyk conditions: $\|K(F_{t_0},x,r,Y)\| \leq r\rho$ for tight certification at time $t$, and $\|K(F_{[t_0,t_0+dt]},x,r,Y)\| \leq r\tau$ for tracking over the interval. The gap between $\tau$ and $\rho$ determines the achievable stepsize.

We take $Y = JF_{t_0}(x)^{-1}$ for the remainder of this section. This simplifies both the analysis and notation, and is the standard choice in practice. With this convention, we write $K(F_{t_0},x,r)$ instead of $K(F_{t_0},x,r,Y)$. We have the first main result as below:

\begin{theorem}\label{thm:stepsize_bound}
Let $F_t:\mathbb{C}^n\times [0,1]\rightarrow\mathbb{C}^n$ be an affine linear parameter homotopy and fix $t_0 \in [0,1]$. Suppose $x$ is a $\rho$-approximate solution to $F_{t_0}=0$ with certification radius $r$ and $Y = JF_{t_0}(x)^{-1}$. Then $x$ is a $\tau$-approximate solution to $F_t=0$ for all $t \in [t_0, t_0+dt]$ whenever
\begin{equation}\label{eq:stepsize_bound}
dt \leq \frac{(\tau-\rho)r}{\|Y F^{(1)}(x)\| + r\sup\limits_{z\in x+rB}\|Y JF^{(1)}(z)\|}.
\end{equation}
\end{theorem}
\begin{proof}
    Since $x$ is a $\rho$-approximate solution to $F_{t_0}=0$, we have
    \begin{equation}\label{eq:rho-approx_condition}
        \|K(F_{t_0},x,r)\|=\|-YF_{t_0}(x)+(Id-Y\square JF_{t_0}(x+rB))rB\|\leq r\rho.
    \end{equation}

    At $t=t_0+dt$, we expand
    \begin{equation*}
        YF_{t_0+dt}(x)=Y(F_{t_0}(x)+dtF^{(1)}(x))
    \end{equation*}
    and for any $z\in x+rB$,
    \begin{align*}
        Id-YJF_{t_0+dt}(z)&=Id-Y(JF_{t_0}(z)+dtJF^{(1)}(z))\\
        &=(Id-YJF_{t_0}(z))-dtYJF^{(1)}(z).
    \end{align*}
    Hence,
    \begin{multline*}
        -YF_{t_0+dt}(x)+(Id-YJF_{t_0+dt}(z))rB \\
        = -YF_{t_0}(x)+(Id-YJF_{t_0}(z))rB 
        - dtY(F^{(1)}(x)+JF^{(1)}(z)rB).
    \end{multline*}
Taking norms and using the inequality \eqref{eq:rho-approx_condition} yields
\begin{multline*}
\|K(F_{t_0+dt},x,r)\|=\|-YF_{t_0+dt}(x)+(Id-YJF_{t_0+dt}(z))rB\|
\\\leq r\rho+dt(\|YF^{(1)}(x)\|+r\|YJF^{(1)}(z)\|).   \end{multline*}

Thus, $x$ is a $\tau$-approximate solution to $F_{t_0+dt}=0$ provided
\[
r\rho+dt\bigl(\|YF^{(1)}(x)\|+r\|YJF^{(1)}(z)\|\bigr)\leq r\tau.
\]
Solving for $dt$ and taking the supremum over $z\in x+rB$ gives the claimed bound.
\end{proof}

The stepsize bound has a natural interpretation. The numerator $(\tau-\rho)r$ measures the slack between certification thresholds, while the denominator captures the rate at which the homotopy varies. Faster variation requires smaller steps. 

Computing the bound \eqref{eq:stepsize_bound} requires interval arithmetic only for $\sup\limits_{z\in x+rB}\|YJF^{(1)}(z)\|$. In contrast, evaluating the entire Krawczyk operator $K(F_t,x,r)$ with interval arithmetic may cause overestimation. Using a \textit{priori} bound avoids this, allows a larger stepsize, and so improves efficiency. Experimental comparisons are provided in \Cref{sec:efficiency}.

We present \Cref{algo:a-priori_track}, which uses the \textit{a priori} stepsize bound from \Cref{thm:stepsize_bound} to eliminate the adaptive tracking of \Cref{algo:track}.

\begin{algorithm}[ht]
	\caption{AprioriTrack}
 \label{algo:a-priori_track}
\begin{algorithmic}[1]
\Require  
a homotopy $F_t:\mathbb{C}^n\times [0,1]\rightarrow \mathbb{C}^n$ with solution path $x(t)$ satisfying $F_t(x(t))=0$, 
a point $x$ approximating $x(0)$ with radius $r$ 
and thresholds $0<\rho<\tau<1$.
\Ensure a $\rho$-approximate solution to $F_1=0$.
\State{Initialize $t=0$.}
\While{$t<1$} 
\State{$x,r,Y=\text{RefineSolution}(F_t,x,r,Y,\rho)$} \Comment{Refine $x$ with $\rho$}
\State{$dt=\frac{(\tau-\rho)r}{\|YF^{(1)}(x)\|+r\sup\limits_{z\in x+rB}\|YJF^{(1)}(z)\|}$}
\State{$t=t+dt$}
\Comment{Proceed with \textit{a priori} stepsize}
\EndWhile
\State{$x,r,Y=\text{RefineSolution}(F_1,x,r,Y,\rho)$} 
\State{\Return $x$}
 \end{algorithmic}
 \end{algorithm}

\section{Complexity analysis}\label{sec:complexity_analysis}

One observes that the stepsize bound \eqref{eq:stepsize_bound} depends on local quantities at each point along the path. By integrating these constraints over the entire solution path $x(t)$ for $t \in [0,1]$, we obtain an upper bound on the total number of tracking steps required.

Define the \emph{weighted path length} as
\[
L := \int_0^1 \frac{1}{r(t)} \left( \|JF_t(x(t))^{-1}F^{(1)}(x(t))\| + r(t) \sup_{z \in x(t)+r(t)B} \|JF_t(x(t))^{-1}JF^{(1)}(z)\| \right) dt
\]
where $r(t)$ denotes the certification radius at time $t$. This quantity captures how the homotopy varies relative to the local certification scale, analogous to the condition-based path length in alpha theory \cite{beltran2013robust,hauenstein2016certified,shub2009complexity}.

The reciprocal of our stepsize bound appears as the integrand, suggesting that the number of steps should be proportional to $L$. We make this precise in this section.

\subsection{Iteration count bound}

We first state an assumption required for the main analysis.
For an affine linear homotopy $F_t=0$ with a nonsingular solution path $x(t)$, differentiating $F_t(x(t))=0$ with respect to $t$ yields
\begin{equation*}
    \frac{dF_t}{dt}+JF_t(x(t))\dot{x}(t)=0.
\end{equation*}
Solving for $\dot{x}(t)$, we have
\begin{equation}\label{eq:velocity}
    \dot{x}(t)=-JF_t(x(t))^{-1}\frac{dF_t}{dt}=-JF_t(x(t))^{-1}F^{(1)}(x(t)).
\end{equation}
Thus $\|JF_t(x(t))^{-1}F^{(1)}(x(t))\|$ represents the speed of the solution path. If this vanishes at some $t$, the path becomes stationary and the complexity analysis breaks down. To exclude this degeneracy, we assume there exists $\eta>0$ such that 
\begin{equation*}
    r(t)\sup\limits_{z\in x(t)+r(t)B}\|JF_t(x(t))^{-1}JF_t(z)\|\leq \eta \|JF_t(x(t))^{-1}F^{(1)}(x(t))\|.
\end{equation*}
This condition ensures the path maintains nonzero velocity throughout $[0,1]$. Also, it is natural to assume it for generic homotopies since $x(t)$ is a nonsingular solution, and so the left-hand side is bounded below by $r(t)$.

We present a useful technical lemma: 
\begin{lemma}
    Let $F_t:\mathbb{C}^n\times [0,1]\rightarrow \mathbb{C}^n$ be a homotopy. For a fixed $t_0\in [0,1]$, let $y$ be a $\rho$-approximate solution to $F_{t_0}=0$ which is also a $\tau$-approximate solution to $F_t=0$ for any $t\in [t_0,t_0+dt]$ for some $dt>0$. Furthermore, let $x(t)$ denote the solution path satisfying $F_t(x(t))=0$ for each $t$. Then, 
    \begin{equation}\label{eq:E-bound}
        \|JF_t(x(t))^{-1}JF_{t_0}(y)\|\geq\frac{1}{1+\tau}. 
    \end{equation}
\end{lemma}
\begin{proof}
    Define $E=Id-JF_{t_0}(y)^{-1}JF_t(x(t))$. As $y$ is a $\tau$-approximate solution to $F_t=0$ for all $t\in [t_0,t_0+dt]$, we have $\|E\|< \tau$. Also, we have $Id-E=JF_{t_0}(y)^{-1}JF_t(x(t))$, and so $(Id-E)^{-1}JF_{t_0}(y)^{-1}=JF_t(x(t))^{-1}$. For a vector $w\in \mathbb{C}^n$, we define $v=(Id-E)^{-1}w$. In this case, 
\begin{equation*}
\|w\|=\|(Id-E)v\|\leq \|v\|+\|Ev\|\leq (1+\|E\|)\|v\|\leq (1+\tau)\|v\|=(1+\tau)\|(Id-E)^{-1}w\|.    
\end{equation*}
It implies that 
\begin{equation*}
    \|(Id-E)^{-1}w\|\geq \frac{1}{1+\tau}\|w\|.
\end{equation*}
The claim follows as $w$ is arbitrary.
\end{proof}

We present the main result of this section.
\begin{theorem}\label{thm:complexity}
    Let $L$ be the weighted path length for a nonsingular solution path $x(t)$ to an affine linear homotopy $F_t=0$. Assume there exists $\eta>0$ such that 
    \begin{equation}\label{eq:eta}
        r(t)\sup\limits_{z\in x(t)+r(t)B}\|JF_t(x(t))^{-1}JF_t(z)\|
        \leq \eta \|JF_t(x(t))^{-1}F^{(1)}(x(t))\|
    \end{equation}
    for all $t\in[0,1]$.
    
    Suppose \Cref{algo:a-priori_track} terminates in $P$ steps with 
    certification radii $r_0, r_1, \ldots, r_{P-1}$ at each refinement. 
    Then,
    \begin{equation}\label{eq:complextiy_bound}
    P \leq \frac{(1+\tau)(1+\eta)}{(\tau-\rho)r_{\min}} L        
    \end{equation}
    where $r_{\min} := \min\limits_{i=0,\dots, P-1}r_i$.
\end{theorem}
\begin{proof}
Suppose that it took $P$ iterations, and so there is a sequence 
\[0=t_0<t_1<\cdots< t_{P-1}< t_P=1.\]   
Fix $i$ and let $\Delta t_i=t_i-t_{i-1}>0$. Let $y_i$ be the $\rho$-approximate solution to $F_{t_i}=0$ with corresponding radius $r$, which is also the $\tau$-approximate solution to $F_t=0$ for all $t\in [t_i,t_{i+1}]$. We lower bound the integral $L$ by replacing continuous path quantities with these discrete approximations obtained from the tracking algorithm. 
Then,
\begin{align*}
    L &= \int_0^1\frac{1}{r(t)}\left(\|JF_t(x(t))^{-1}F^{(1)}(x(t))\|+r(t)\sup\limits_{z\in x(t)+r(t)B}\|JF_t(x(t))^{-1}JF^{(1)}(z)\|\right)dt\\
    &=\sum_{i=0}^{P-1}\int_{t_i}^{t_{i+1}}\frac{1}{r(t)}\left(\|JF_t(x(t))^{-1}F^{(1)}(x(t))\|+r(t)\sup\limits_{z\in x(t)+r(t)B}\|JF_t(x(t))^{-1}JF^{(1)}(z)\|\right)dt\\
    &\geq \sum_{i=0}^{P-1}\Delta t_i\min\limits_{t_i\leq t\leq t_{i+1}}\frac{1}{r(t)}\left(\|JF_t(x(t))^{-1}F^{(1)}(x(t))\|+r(t)\sup\limits_{z\in x(t)+r(t)B}\|JF_t(x(t))^{-1}JF^{(1)}(z)\|\right)\\
    &\geq \sum_{i=0}^{P-1}\Delta t_i\min\limits_{t_i\leq t\leq t_{i+1}}\|JF_t(x(t))^{-1}F^{(1)}(x(t))\|+r(t)\sup\limits_{z\in x(t)+r(t)B}\|JF_t(x(t))^{-1}JF^{(1)}(z)\|.
\end{align*}
The last inequality uses $r(t) \leq 1$, which is ensured by \Cref{algo:meta_refine}.

First, note that 
\begin{align*}
    \|JF_t(x(t))^{-1}F^{(1)}(x(t))\|&=\|JF_t(x(t))^{-1}(F^{(1)}(y_i)-F^{(1)}(y_i)+F^{(1)}(x(t)))\|\\
    &\geq \|JF_t(x(t))^{-1}F^{(1)}(y_i)\|-\|JF_t(x(t))^{-1}(F^{(1)}(y_i)-F^{(1)}(x(t)))\|\\
    &\geq \|JF_t(x(t))^{-1}F^{(1)}(y_i)\|-r(t)\sup\limits_{z\in x(t)+r(t)B}\|JF_t(x(t))^{-1}JF^{(1)}(z)\|.
\end{align*}
Therefore, 
\begin{equation}\label{eq:step-bound}
    \|JF_t(x(t))^{-1}F^{(1)}(x(t))\|+r(t)\sup\limits_{z\in x(t)+r(t)B}\|JF_t(x(t))^{-1}JF^{(1)}(z)\|\geq \|JF_t(x(t))^{-1}F^{(1)}(y_i)\|.
\end{equation}

Combining inequalities (\ref{eq:E-bound}) and (\ref{eq:step-bound}) yields that
\begin{equation*}
    \|JF_t(x(t))^{-1}F^{(1)}(y_i)\|=\|(JF_t(x(t))^{-1}JF_{t_i}(y_i))JF_{t_i}(y_i)^{-1}F^{(1)}(y_i)\|\geq\frac{1}{1+\tau}\|JF_{t_i}(y_i)^{-1}F^{(1)}(y_i)\|.
\end{equation*}
From the assumption that 
\begin{equation*}
    r_i\sup_{z\in y_i+r_iB}\|JF_{t_i}(y_i)^{-1}JF_{t_i}(z)\|\leq \eta \|JF_{t_i}(y_i)^{-1}F^{(1)}(y_i)\|,
\end{equation*}
we have 
\begin{equation*}
    \Delta t_i \geq \frac{(\tau-\rho)r_i}{(1+\eta)\|JF_{t_i}(y_i)^{-1}F^{(1)}(y_i)\|}
\end{equation*}
which implies that
\begin{equation*}
    \Delta t_i\|JF_{t_i}(y_i)^{-1}F^{(1)}(y_i)\|\geq \frac{(\tau-\rho)r_i}{1+\eta}.
\end{equation*}
Therefore, 
\begin{equation*}
    L\geq \sum_{i=0}^{P-1}\Delta t_i\|JF_t(x(t))^{-1}F^{(1)}(y_i)\|\geq \sum_{i=0}^{P-1}\frac{\Delta t_i\|JF_{t_i}(y_i)^{-1}F^{(1)}(y_i)\|}{1+\tau}\geq \sum_{i=0}^{P-1}\frac{(\tau-\rho)r_i}{(1+\tau)(1+\eta)}
\end{equation*}
and the claimed result follows.
\end{proof}

\subsection{Certification radius bounds via Smale's alpha theory}

\Cref{thm:complexity} proposes a complexity bound in terms of the minimum certification radius $r_{\min}$. However, this approach is inherently conservative as it ignores 
the local behavior of the homotopy and treats all steps as if they achieve only the worst-case certification 
radius. Hence, it is natural to find the bound of each $r_i$ using computable quantities. We derive such bounds using Smale's alpha theory.

Let $F:\mathbb{C}^n\to\mathbb{C}^n$ be a system and $x\in\mathbb{C}^n$. 
Smale's alpha theory introduces two 
fundamental quantities:
\begin{equation*}
\beta(F,x) := \|JF(x)^{-1}F(x)\|, \quad
\gamma(F,x) := \sup_{k\geq 2}\|\frac{JF(x)^{-1}J^kF(x)}{k!}\|^{\frac{1}{k-1}}. 
\end{equation*}
The quantity $\beta(F,x)$ measures 
the size of the Newton step, while $\gamma(F,x)$ captures the local nonlinearity 
of $F$ near $x$ through its higher derivatives. Together, they determine 
the convergence behavior of Newton's method and the existence of nearby solutions. For more details, see \cite[Chapter 8]{blum2012complexity}. The following proposition provides a lower bound on 
certification radii using these quantities.

\begin{prop}\label{prop:radii_bound}
Let $F_t:\mathbb{C}^n\times[0,1]\to\mathbb{C}^n$ be a homotopy, 
and let $y_i$ be a $\rho$-approximate solution to $F_{t_i}=0$ 
with certification radius $r_i$. Assume that 
$u := \gamma(F_{t_i},y_i) r_i \leq 1-\frac{\sqrt{2}}{2}$. Then
\begin{equation}\label{eq:r_theory}
r_i \geq \frac{\beta(F_{t_i},y_i)}{\rho+1}.    
\end{equation}
\end{prop}
\begin{proof}
Consider the Krawczyk operator $K(F_{t_i},y_i,r_i)$. Then,
\begin{align*}
\|K(F_{t_i},y_i,r_i)\|
&= \|-JF_{t_i}(y_i)^{-1}F_{t_i}(y_i)+(Id-JF_{t_i}(y_i)^{-1}JF_{t_i}(y_i+r_iB))r_iB\|\\
&= \|-JF_{t_i}(y_i)^{-1}F_{t_i}(y_i)+JF_{t_i}(y_i)^{-1}(JF_{t_i}(y_i)-JF_{t_i}(y_i+r_iB))r_iB\|\\
&\geq \|JF_{t_i}(y_i)^{-1}F_{t_i}(y_i)\| -\|JF_{t_i}(y_i)^{-1}\left(JF_{t_i}(y_i)-\sum_{k=1}^\infty\frac{J^kF_{t_i}(y_i)}{(k-1)!}(r_iB)^{k-1}\right)r_iB\|\\
&= \|JF_{t_i}(y_i)^{-1}F_{t_i}(y_i)\|-\|JF_{t_i}(y_i)^{-1}\sum_{k=2}^\infty\frac{J^kF_{t_i}(y_i)}{(k-1)!}r_i^{k-1}\|r_i\\
&= \|JF_{t_i}(y_i)^{-1}F_{t_i}(y_i)\|-\|JF_{t_i}(y_i)^{-1}\sum_{k=2}^\infty k\frac{J^kF_{t_i}(y_i)}{k!}r_i^{k-1}\|r_i\\
&= \beta(F_{t_i},y_i)-\sum_{k=2}^\infty k\gamma(F_{t_i},y_i)^{k-1}r_i^k\\
&= \beta(F_{t_i},y_i)-\left(\frac{1}{(1-u)^2}-1\right)r_i\\
&\geq \beta(F_{t_i},y_i)-r_i,
\end{align*}
the last inequality uses 
the assumption $u \leq 1-\frac{\sqrt{2}}{2}$, which implies 
$\frac{1}{(1-u)^2}-1 \leq 1$.

Since $y_i$ is a $\rho$-approximate solution to $F_{t_i}=0$, 
we have $\|K(F_{t_i},y_i,r_i)\| \leq r_i\rho$. Therefore, $\beta(F_{t_i},y_i) - r_i \leq r_i\rho$,
which yields $r_i \geq \frac{\beta(F_{t_i},y_i)}{\rho+1}$.
\end{proof}

The bound in Proposition \ref{prop:radii_bound} is typically conservative since the triangular inequality in the proof ignores cancellation in the Krawczyk operator. The resulting bound is proportional to $\beta(F_{t_i}, y_i)$, yet an assumption that $y_i$ is $\rho$-approximate implies $\beta$ is small. As demonstrated in Section~\ref{sec:empirical_analysis}, actual certification radii greatly exceed this bound. This highlights the need for tighter analysis of Krawczyk-based certification.

\section{Experiments}\label{sec:experiments}

We implemented \Cref{algo:a-priori_track} in \texttt{Julia}, utilizing the \texttt{Nemo.jl} library \cite{fieker2017nemo} for rigorous complex interval arithmetic. Throughout this section, we use $\rho=\frac{1}{8}$ and $\tau=\frac{7}{8}$.

\subsection{Benchmark examples}\label{sec:empirical_analysis}

The primary goal of this experiment is to investigate the practical behavior of the quantities introduced in our main results. While the complexity bounds in \Cref{sec:a_priori_predictor,sec:complexity_analysis} provide worst-case guarantees, we aim to quantify their magnitude and distribution in actual tracking scenarios. Additionally, we report auxiliary quantities obtained from the implementation to provide a complete picture of the algorithm's performance.

We consider two families of benchmark systems: Katsura systems and dense random systems of degree $2$ with varying the number of variables. For each system, we construct a linear homotopy starting from a B\'ezout start system and track all solution paths. We report the average of the following quantities: \begin{itemize} 
\item the minimum stepsize;
\item the median stepsize for each solution path, which shows the typical behavior of the stepsize;
\item the minimum certification radius $r_{\min}$ for each solution path, which is the critical factor in the complexity bound \eqref{eq:complextiy_bound}; \item The gap between the theory and practice, measured by the average ratio between the certification radius $r_i$ and its theoretical lower bound \eqref{eq:r_theory} at each iteration; and
\item The maximum value of $\eta$ defined in \eqref{eq:eta}, to verify the regularity assumption of the solution paths. 
\end{itemize} \Cref{tab:data_stat} summarizes the results.

\begin{table}[ht]
\centering
\renewcommand{\arraystretch}{1.20}
\begin{tabular}{c||c|c|c|c|c|c|c}  
system & paths & iters &  $dt_{\min}$ & median $dt$ & $r_{\min}$ & avg. ($\frac{r_i}{r_{\text{theory}}}$) & $\eta_{\max}$\\
\hline
katsura3 & $4$ & $1101.25$ & 
$4.2\cdot10^{-4}$ & 
$9.6\cdot 10^{-4}$ &
$1.6\cdot 10^{-3}$ & 
$1.9\cdot10^{5}$ &
$2.1\cdot10^{-2}$ \\
katsura4 & $8$ & $1942.25$ & 
$1.5\cdot10^{-4}$ & 
$4.9\cdot 10^{-4}$ &
$9.8\cdot 10^{-4}$ &
$1.5\cdot 10^{5}$ & 
$1.9\cdot10^{-2}$\\
katsura5 & $16$ & $6892.81$ & 
$2.9\cdot10^{-5}$ & 
$1.4\cdot 10^{-4}$ &
$4.4\cdot 10^{-4}$ & 
$6.4\cdot10^{4}$ & 
$1.3\cdot10^{-2}$\\
katsura6 & $32$ & $12777.9$ & 
$2.8\cdot10^{-5}$ & 
$1.3\cdot 10^{-4}$ &
$2.7\cdot 10^{-4}$ & 
$1.4\cdot10^{5}$ & 
$1.3\cdot10^{-2}$\\
random$_{(2^3)}$ & $8$ & $2498.5$ & 
$7.4\cdot10^{-4}$ & 
$1.2\cdot 10^{-3}$ &
$1.7\cdot 10^{-3    }$ & 
$3.0\cdot10^{5}$ & 
$3.0\cdot10^{-2}$\\
random$_{(2^4)}$ & $16$ & $8275.44$ & 
$2.4\cdot10^{-4}$ & 
$5.1\cdot 10^{-4}$ &
$9.7\cdot 10^{-4}$ & 
$1.0\cdot10^{5}$ & 
$2.7\cdot10^{-2}$\\
random$_{(2^5)}$ & $32$ & $36323.3$ & 
$3.9\cdot10^{-5}$ & 
$8.6\cdot 10^{-5}$ &
$2.7\cdot 10^{-4}$ & 
$1.0\cdot10^{5}$ & 
$5.6\cdot10^{-2}$\\
random$_{(2^6)}$ & $64$ & $105884.52$ & 
$2.6\cdot10^{-5}$ & 
$6.3\cdot 10^{-5}$ &
$2.6\cdot 10^{-4}$ & 
$6.5\cdot10^{4}$ & 
$2.4\cdot10^{-2}$

\end{tabular}
\smallskip
\caption{Average iterations per path, stepsize 
statistics, minimum certification radius, theory-practice gap 
(Proposition \ref{prop:radii_bound}), and the maximum $\eta$ parameter 
(Proposition \ref{eq:eta}).}\label{tab:data_stat}

\end{table}

Relying only on averages across all solution paths can obscure the actual tracking behaviors, as it hides the impact of outliers and difficult paths. To address this, \Cref{fig:boxplot} presents box plots of the stepsize distribution of a single solution path from Katsura 4, 5, 6 and random systems with $4,5,6$ equations.

We observe a significant gap between the theoretical certification radius and the actual radius used in practice. The average ratio $\frac{r_i}{r_{\text{theory}}}$ consistently lies in the range of $10^4$ to $10^5$. This indicates that the theoretical bound derived in Proposition \ref{prop:radii_bound} is rigorous but conservative compared to the computed Krawczyk regions.
Also, the maximum values of $\eta$ remain small, approximately $10^{-2}$ across all benchmarks. It validates the regularity assumption \eqref{eq:eta} and shows that the total iteration count $P$ is significantly governed by $L$ and $r_{\min}$. Lastly, as the dimension of the system increases (e.g., from Katsura 4 to 6), both $dt_{\min}$ and the median stepsize decrease, while the spread of the distribution widens. As shown in \Cref{fig:boxplot} (plotted on a $\log_{10}$ scale), Katsura 6 exhibits a much larger variance in stepsize compared to smaller systems. This suggests that higher-dimensional paths may encounter more complex local geometries, and they necessitate reductions in stepsize to maintain certification, which impacts the total number of iterations.

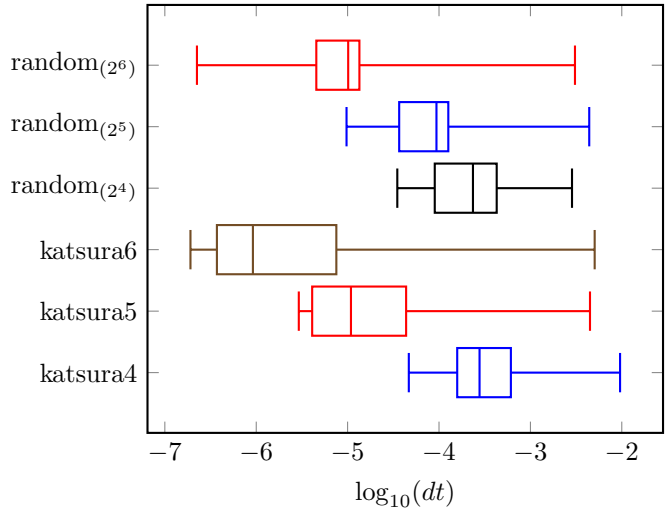
\begin{figure}[H]
    \centering
    \begin{tikzpicture}
  \begin{axis}
    [
    thick=.5,
    xlabel={$\log_{10}(dt)$},
    ytick={1,2,3,4,5,6},
    yticklabels={katsura4, katsura5, katsura6,random$_{(2^4)}$,random$_{(2^5)}$,random$_{(2^6)}$}
    ]
    \addplot+[
    boxplot prepared={
      median= -3.5564767719499866,
      upper quartile= -3.213556691871661,
      lower quartile= -3.800882868023702,
      upper whisker=-2.0189159613361665,
      lower whisker= -4.330409895965605
    },
    ] coordinates {};
    \addplot+[
    boxplot prepared={
      median=-4.963473447904334,
      upper quartile= -4.359926301025817,
      lower quartile=-5.387952152232032,
      upper whisker= -2.3479329351219187,
      lower whisker=-5.534190753729732
    },
    ] coordinates {};
    \addplot+[
    boxplot prepared={
      median= -6.0363538746381495,
      upper quartile= -5.123403505053829,
      lower quartile=-6.430240542264745,
      upper whisker=  -2.296733182773052,
      lower whisker= -6.719686311999436
    },
    ] coordinates {}; 
        \addplot+[
    boxplot prepared={
      median= -3.6285755728822955,
      upper quartile= -3.3687325465765716,
      lower quartile= -4.046587324373451,
      upper whisker=-2.5447082182938323,
      lower whisker= -4.456410032417957
    },
    ] coordinates {};
            \addplot+[
    boxplot prepared={
      median= -4.0279136578740244,
      upper quartile= -3.8983309833957964,
      lower quartile= -4.43584982704036,
      upper whisker=-2.3562627472269573,
      lower whisker= -5.011864244006888
    },
    ] coordinates {};
            \addplot+[
    solid,boxplot prepared={
      median= -4.99384154564455,
      upper quartile= -4.871030320268869,
      lower quartile= -5.342599527778019,
      upper whisker= -2.5122964897747115,
      lower whisker= -6.648709940109651
    },
    ] coordinates {};
    \end{axis}
\end{tikzpicture}

    \caption{Stepsize distributions (log scale) for a single solution path from each benchmark system. The boxes show interquartile ranges with the median marked by the vertical lines. Whiskers extend to the minimum and maximum values.}
    \label{fig:boxplot}
\end{figure}

\subsection{Univariate examples: theory validation via exact computation}

In \Cref{sec:complexity_analysis}, the number of iterations \eqref{eq:complextiy_bound} is proportional to the weight path length $L$. To validate this directly, we construct a simple univariate system where both quantities can be computed exactly.

Consider the target system $f(x) = x^2 - m = 0$ for a constant $m > 1$, with the B\'ezout start system $g(x) = x^2 - 1 = 0$. The 
linear homotopy is $F_t = (1-t)g + tf = 0$, which simplifies to 
$F_t(x) = x^2 - mt + t - 1 = 0$. The positive solution path is 
$x(t) = \sqrt{mt - t + 1}$.

We compute the weighted path length
\[
L := \int_0^1 \frac{1}{r(t)} \left( \|JF_t(x(t))^{-1}F^{(1)}(x(t))\| 
+ r(t) \sup_{z \in x(t)+r(t)B} \|JF_t(x(t))^{-1}JF^{(1)}(z)\| \right) dt.
\]
By \Cref{eq:velocity}, the first term equals the path velocity
\begin{equation*}
\|JF_t(x(t))^{-1}F^{(1)}(x(t))\| = \|\dot{x}(t)\| 
= \left|\frac{m-1}{2\sqrt{mt-t+1}}\right|.
\end{equation*}
Note that $JF^{(1)} = 0$ since $F^{(1)}(x) = -(m-1)$ is constant, and therefore the second term vanishes.
Moreover, the certification radius remains constant. Tracking $F_t=0$ 
confirms $r(t) = 0.05$ for all $t \in [0,1]$, as the Jacobian 
$JF_t(x) = 2x$ varies smoothly along the path. Combining these, 
we obtain
\begin{equation*}
L = \frac{1}{0.05} \int_0^1 \frac{m-1}{2\sqrt{mt-t+1}} dt 
  = 20 \int_0^1 \frac{m-1}{2\sqrt{1+t(m-1)}} dt 
  = 20(\sqrt{m} - 1).
\end{equation*}
\Cref{tab:iter_per_length} reports the ratio of actual 
iteration count to computed path length for various values of $m$.

\begin{table}[ht]
\centering
\renewcommand{\arraystretch}{1.20}
\begin{tabular}{c||c}  
$m$ & iters$/L$ \\
\hline
$10$ & $1.3643$ \\
$100$ & $1.3444$ \\
$1000$ & $1.3372$ \\
$10000$ & $1.3348$ \\
$20000$ & $1.3346$ \\
$30000$ & $1.3342$

\end{tabular}
\smallskip
\caption{Ratio of iteration count to weighted path length for univariate quadratic systems.}\label{tab:iter_per_length}

\end{table}
With uniform certification radius $r(t) = r_{\min} = 0.05$ throughout the path, the example confirms that the ratio between the number of iterations $P$ and the length of the path $L$ are approximately $1.33$ across all test cases. It demonstrates the linear relationship between $P$ and $L$.

\subsection{Computational efficiency}\label{sec:efficiency}

Another takeaway is that reducing 
interval arithmetic improves tracking efficiency. 
Algorithm~\ref{algo:a-priori_track} demonstrates this  
for constant predictors by eliminating adaptive stepsize selection, 
which requires repeated Krawczyk operator evaluation with interval 
arithmetic. We investigate the practical impact of interval arithmetic reduction 
on tracking efficiency across constant and higher-order predictors.

We compare the following algorithms on benchmark systems with B\'ezout start systems by tracking all solution paths and reporting 
average iteration counts:

\begin{itemize}
\item \textbf{Adaptive} (\Cref{algo:track}): 
      Constant predictor with adaptive stepsize selection via 
      repeated interval arithmetic evaluation.

\item \textbf{\textit{A priori}} (\Cref{algo:a-priori_track}): 
      Constant predictor using \Cref{thm:stepsize_bound} to 
      avoid adaptive selection.

\item \textbf{Higher-order} \cite[Algorithm 4]{guillemot2024validated}: 
      Hermite predictor with adaptive stepsize. Uses Taylor models \cite{neumaier2003taylor} to reduce overestimation 
      in the predictor step.

\item \textbf{Higher-order without truncation}: 
      Hermite predictor without Taylor models. It demonstrates 
      interval arithmetic overhead when Taylor models are not used.

\item \textbf{Mitigated higher-order}: Higher-order without truncation, but with
the inequality \eqref{eq:krawczyk_condition} replaced by \begin{equation*}
      \|-YF_{t_0}(x)\| + r\|(Id-Y\square JF_{t_0}(x+rB))\| \leq r\rho
      \end{equation*}
      to reduce interval arithmetic at the 
      certification step.

\item \textbf{Alpha theory tracking} \cite{beltran2011continuation}: 
      For comparison with known certified tracking methods.
\end{itemize}

\begin{table}[ht]
\centering
\renewcommand{\arraystretch}{1.20}
\begin{tabular}{c||c|c|c|c|c|c|c}  
System & Paths & Algo. \ref{algo:track} & Algo. \ref{algo:a-priori_track} & \begin{tabular}{@{}c@{}} Truncated \\ Hermite\end{tabular} & \begin{tabular}{@{}c@{}} Non-truncated\\Hermite \end{tabular}   & \begin{tabular}{@{}c@{}} Mitigated\\Hermite\end{tabular}& BL \cite{beltran2012certified}\\
\hline
katsura3 & $4$ & $2728.75$& 
$1978.75$  & 
$51.5$ &
$60.0$ & 
$44.5$ &
$569.5$ \\
katsura4 & $8$ & $4685.63$ & 
$3515.38$ & 
$58.5$ &
$70.25$ &
$55.0$ & 
$1149.88$\\
katsura5 & $16$ & $13148.0$ & 
$9774.94$ & 
$92.57$ &
$107.75$ & 
$78.25$ & 
$1498.38$\\
katsura6 & $32$ & $13229.5$ & 
$9949.09$ & 
$115.25$ &
$138.5$ & 
$108.91$ & 
$2361.81$\\
random$_{(2^3)}$ & $8$ & $3457.5$ & 
$2590.63$ & 
$32.25$ &
$38.63$ & 
$31.5$ & 
$370.13$\\
random$_{(2^4)}$ & $16$ & $54288.19$ & 
$40098.81$ & 
$80.44$ &
$102.25$ & 
$84.5$ & 
$813.81$\\
random$_{(2^5)}$ & $32$ & $44049.28$ & 
$33112.31$ & 
$56.78$ &
$72.69$ & 
$59.47$ & 
$1542.5$\\
random$_{(2^6)}$ & $64$ & $58106.27$ & 
$42604.41$ & 
$92.21875$ &
$115.66$ & 
$93.03$ & 
$2361.81$

\end{tabular}
\smallskip
\caption{Average iteration counts per solution path comparing constant predictor methods (\Cref{algo:track,algo:a-priori_track}), 
higher-order Krawczyk predictors, 
and alpha theory tracking \cite{beltran2012certified}.}\label{tab:ia_reduction}

\end{table}

Comparing \Cref{algo:track} with \Cref{algo:a-priori_track}, we observe improvement from all benchmarks. Improvements from reducing interval arithmetic were even observed for higher-order predictors in several test cases. While we analyze only constant predictors in \Cref{sec:a_priori_predictor,sec:complexity_analysis}, the core insight (minimizing interval arithmetic overhead) applies and yields practical benefits from different predictor types.

Extending \textit{a priori} stepsize to higher-order predictors remains a future direction. The predictor step involves higher derivatives whose interval enclosures are more complex, making stepsize bound derivation nontrivial. However, the mitigated Hermite results suggest such an extension could result in reduced overhead and advance the efficiency of certified tracking.

\section*{Acknowledgements}

We thank Michael Burr, Alexandre Guillemot, and Jonathan Hauenstein for helpful discussions.

\bibliographystyle{abbrv}
\bibliography{ref.bib}
\end{document}